\documentclass[12pt]{amsart}
\usepackage{amsmath,amssymb,amsthm,amsfonts,enumerate}

\usepackage{float}
\usepackage{graphicx}
\usepackage{setspace}
\usepackage{MnSymbol}
\usepackage{amscd}
\usepackage{relsize}
\usepackage{hyperref}
\usepackage{xcolor}

\newtheorem{thm}{Theorem}[section]

\newtheorem{lem}[thm]{Lemma}
\newtheorem{cor}[thm]{Corollary}

\newcommand{\M}{\mathcal{M}}
\newcommand{\T}{\mathcal{T}}

\newcommand{\Z}{\mathbb{Z}}

\theoremstyle{definition}

\theoremstyle{definition}

\theoremstyle{remark}

\title[Generating the mapping class group by three torsions]
{Generating the mapping class group of a
nonorientable surface by three torsions}
\author{Marta Le\'sniak}
\author{B{\l}a\.zej Szepietowski}
\address{Institute of Mathematics, Faculty of Mathematics, Physics and Informatics, University of Gda\'nsk, 80-308 Gda\'nsk, Poland} 
\email{Marta.Lesniak@mat.ug.edu.pl}
\email{blaszep@mat.ug.edu.pl}

\begin{document}
\begin{abstract}
We prove that the mapping class group $\mathcal{M}(N_g)$ of a closed nonorientable surface of genus $g$ different than 4 is generated by three torsion elements. Moreover, for every even integer $k\ge 12$ and $g$ of the form $g=pk+2q(k-1)$ or $g=pk+2q(k-1)+1$, where $p,q$ are non-negative integers and $p$ is odd,   $\M(N_g)$ is generated by three conjugate elements of order $k$. Analogous results are proved for the subgroup of $\mathcal{M}(N_g)$  generated by Dehn twists.
\end{abstract}

\maketitle

\section{Introduction}
Let $\M(S_g)$ and $\M(N_g)$ denote mapping class groups of respectively orientable and nonorientable closed surface of genus $g$. It is well known that these groups are generated by elements of finite order. This fact was first proved for $\M(S_g)$ by Maclachlan \cite{McL2}, who used this result in his proof of simple connectivity of the moduli space of Riemann surfaces of genus $g$. Analogous result for $\M(N_g)$  was obtained by the second author \cite{Sz2}.

It is also known that $\M(S_g)$ is normally generated by one torsion element for $g\ge 3$, that is there exits 
$f\in\mathcal{M}(S_g)$ of finite order whose normal closure in $\mathcal{M}(S_g)$ is the entire group.  McCarthy-Papadopoulos \cite{MP} proved that $\mathcal{M}(S_g)$ is generated by conjugates of one involution, and Korkmaz \cite{MK2} showed that it is generated by two conjugate elements of order $4g+2$. More recently, Lanier and Margalit \cite{LM} proved that any non-trivial periodic mapping class other than the hyperelliptic involution is a normal generator of $\mathcal{M}(S_g)$ for $g\ge 3$. Lanier \cite{Lanier} showed that for any $k\ge 6$ and $g$ sufficiently large, $\M(S_g)$ can be generated by three conjugate elements of order $k$.

The first author \cite{ML} showed that also $\M(N_g)$ is normally generated by one torsion element if $g\ge 7$. Moreover, if $f\in\M(N_g)$ has finite order at least $3$, then the normal closure of $f$ in $\M(N_g)$ is either the entire mapping class group or its index two subgroup generated by Dehn twists, denoted by $\T(N_g)$ and called the {\it twist subgroup}. The second author \cite{SzGD} proved that $\M(N_g)$ is generated by three elements and by four involutions. Altun\"oz, Pamuk and Yildiz showed that $\T(N_g)$ is generated by  three torsion elements for $g\geq 13$, out of which two are involutions, while the order of the third one  depends on $g$ \cite{APY}. Similar result for odd $g$ was obtained by Du \cite{Du}.
The aim of this paper is to prove the following results.
\begin{thm}\label{thm1}
For $g\neq 4$, $\mathcal{M}(N_g)$ and $\mathcal{T}(N_g)$ are generated by three torsion elements.
\end{thm}
By \cite[Proposition 7.16]{MS} the abelianization of $\T(N_4)$ is $\Z_2\times\Z$, and it follows that $\T(N_4)$ is not generated by torsion elements. On the other hand, the abelianization of $\M(N_4)$ is $\Z_2^3$ \cite{KorkH1}, but we do not know whether $\M(N_4)$ can be generated by three torsions.
The orders of the three torsion generators of $\mathcal{M}(N_g)$ and $\mathcal{T}(N_g)$ given in the proof of Theorem \ref{thm1} depend on $g$. Our next result shows that for sufficiently large $g$, $\M(N_g)$ and $\T(N_g)$ are generated by three conjugate torsion elements. It is inspired by Lanier's work \cite{Lanier}.

\begin{thm}\label{thm2}
Let $k\geq 12$ and  $g=pk+2q(k-1)$ or $g=pk+2q(k-1)+1$, where $p,q\in\mathbb{Z}_{\geq 0}$ and $p\ge 1$. There exist three conjugate elements of $\M(N_g)$ of order $k$ such that the subgroup of $\mathcal{M}(N_g)$  generated by these elements is either $\T(N_g)$ or $\M(N_g)$, the latter if and only if $k$ is even and $p$ is odd.
\end{thm}
In particular, by setting $p=1$ and $q=0$ in Theorem \ref{thm2} we obtain the following.

\begin{cor}
(a) For even $g\ge 12$, $\M(N_g)$ is generated by three conjugate elements of order $g$, and $\T(N_g)$ is generated by three  elements of order $g-1$ conjugate in $\M(N_g)$.

(b) For odd $g\ge 13$, $\M(N_g)$ is generated by three conjugate elements of order $g-1$, and $\T(N_g)$ is generated by three  elements of order $g$ conjugate in $\M(N_g)$.
\end{cor}
Since $\M(N_g)$ contains a subgroup of index two, namely the twist subgroup, it cannot be generated by elements of odd order. We also remark that for $2\le g\le 6$, $\M(N_g)$ is not normally generated by a single element, since its abelianization is not cyclic \cite{KorkH1}.

\begin{cor}\label{cor1}
	If $k\geq 12$ is even and $g\geq 2(k-1)(k-2)+k$, then there exists a set of three conjugate elements of order $k$ generating $\mathcal{M}(N_g)$.
\end{cor}

\section{Preliminaries}

Let $N_g$ be a nonorientable surface of genus $g$.
We represent $N_g$ as a sphere with $g$ crosscaps (Figure \ref{fig:genkicurves}). In all figures of this paper a shaded, crossed disc represents a crosscap.
This means that the interior of each such disc should be removed, and then antipodal points on the resulting boundary component should be identified.  
The mapping class group $\M(N_g)$ is the group of isotopy classes of homeomorphisms of $N_g$.

By a curve in $N_g$ we understand a simple closed curve. We do not distinguish a curve from its isotopy class. A curve is two-sided if its regular neighbourhood is an annulus and one-sided if it is a M\"obius strip. Moreover, a curve $\alpha$ is nonseparating if its complement $N\setminus\alpha$ is connected and separating otherwise.

If $\alpha$ and $\beta$ are curves, we denote by $i(\alpha,\beta)$ their geometric intersection number.

A sequence of two-sided curves $(\alpha_1,\dots,\alpha_n)$ is called an $n$-chain if for every $1\leq i\leq n-1$, $i(\alpha_{i},\alpha_{i+1})=1$ and $i(\alpha_{i},\alpha_j)=0$ for $|i-j|\geq 2$.
Let $C=(\alpha_i)_{i=1}^n$ be an $n$-chain on $N$.
By a regular neighbourhood of $C$ we understand a regular neighbourhood of $\bigcup_{i=1}^n\alpha_i$.
We say that $C$ is nonseparating if $N\setminus\bigcup_{i=1}^n\alpha_i$ is connected. Note that a chain of even length is always nonseparating.

If $\alpha$ is a  two-sided curve on $N_g$, denote by $t_\alpha$ a Dehn twist about $\alpha$. Because we are working on a nonorientable surface, the Dehn twist $t_\alpha$ can be one of two, depending on which orientation of a regular neighbourhood of $a$ we choose. Then the Dehn twist in the opposite direction is $t_\alpha^{-1}$. For any $f\in\mathcal{M}(N_g)$ we have

\begin{equation*}
ft_\alpha f^{-1}=t_{f(\alpha)}^\epsilon
\end{equation*}
where $\epsilon\in\{-1,1\}$, again depending on which orientation of a regular neighbourhood of $f(\alpha)$ we choose. The subgroup of $\mathcal{M}(N_g)$ generated by all Dehn twists has index two. We call it the twist subgroup and denote by $\mathcal{T}(N_g)$. 

Consider the action of $\mathcal{M}(N_g)$ on $H_1(N_g;\mathbb{R})$. For $f\in\mathcal{M}(N_g)$ let $f_*:H_1(N_g;\mathbb{R})\rightarrow H_1(N_g;\mathbb{R})$ be the induced homomorphism.

\begin{thm}\label{thm:twsb}\cite{MS}
For $f\in \mathcal{M}(N_g)$, $\det(f_*)=1$ if and only if $f\in\mathcal{T}(N_g)$.
\end{thm}

\begin{figure}
\begin{center}
\includegraphics[scale=.4]{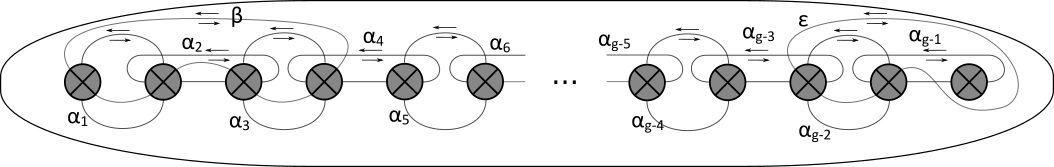}
\caption{Curves corresponding to generators of the twist subgroup $\T(N_g)$.}
\label{fig:genkicurves}
\end{center}
\end{figure}

Building on Stukow's results \cite{MS,MS2}, Omori \cite{Genki} proved that $\mathcal{T}(N_{g})$ is generated by  $g+1$ Dehn twists.
\begin{thm}\cite{Genki}\label{thm:Genki}
For $g\geq 4$, $\mathcal{T}(N_{g})$ is generated by Dehn twists about the curves $\alpha_1$,..., $\alpha_{g-1}$, $\beta$ and $\epsilon$ shown on Figure \ref{fig:genkicurves}.
\end{thm}
The curves on Figure \ref{fig:genkicurves} are not exactly the same as in \cite{Genki}, but by applying 
$(t_{\alpha_1}t_{\alpha_2}\cdots t_{\alpha_{g-1}})^{g-4}$ to $\beta$, we can ``shift'' it to the right-hand side of the figure and obtain the same configuration of curves as in \cite{Genki}, up to relabelling. 
Note that $(\alpha_i)_{i=1}^{g-1}$ is a $(g-1)$-chain, $\beta$ is disjoint from $\alpha_i$ for $i\ne 4$ and $i(\beta,\alpha_4)=1$ if $g\ge 5$.
We fix Dehn twists  $t_{\alpha_i}$, $i=1,\dots,g-1$ and $t_\beta$ in directions indicated by the arrows in Figure \ref{fig:genkicurves}.

\begin{figure}\begin{center}
\includegraphics[scale=.5]{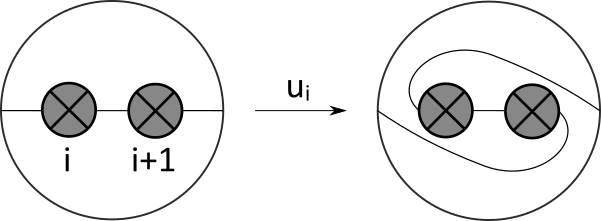}
\caption{The crosscap transposition $u_i$.}
\label{fig:transpo}
\end{center}
\end{figure}

For $i=1,\dots,g-1$ we define a {\it crosscap transposition} $u_i$ to be the isotopy class of a homeomorphism interchanging two consecutive crosscaps as shown on Figure \ref{fig:transpo} and equal to the identity outside a disc containing these crosscaps. It follows immediately from Theorem \ref{thm:twsb} that $u_i\notin\T(N_g)$.

\begin{figure}
\begin{center}
	\includegraphics[scale=.5]{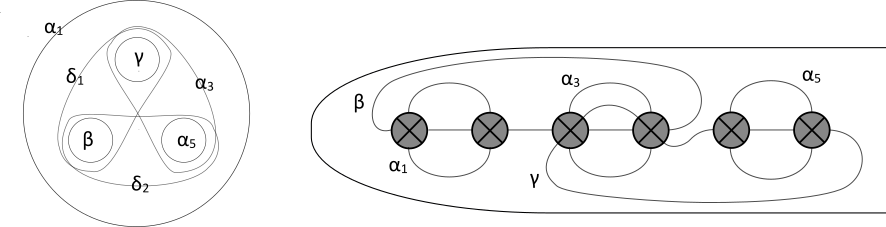}
	\caption{Curves of the lantern relation on $L$.}
	\label{lantern1}	\end{center}
\end{figure}

\medskip

Let $L\subset N_g$ be the four-holed sphere bounded by the curves $\beta$, $\alpha_1$, $\alpha_5$ and $\gamma$ shown on Figure \ref{lantern1} and oriented so that the twists $t_{\beta}$, $t_{\alpha_1}$ and $t_{\alpha_5}$ are right-handed. If $\delta_1$ and $\delta_2$ are curves on $L$ as on the left-hand side of Figure \ref{lantern1}, then the following \emph{lantern relation} between right-handed Dehn twists holds (see \cite{FM}):
	\[
t_{\alpha_1}t_{\beta}t_{\gamma}t_{\alpha_5}=t_{\alpha_3}t_{\delta_1}t_{\delta_2}.
	\]
Note that $\delta_1$ and $\delta_2$ are not unique. Indeed,  applying $t_{\alpha_3}$ changes $\delta_1$ and $\delta_2$ keeping the other curves of Figure \ref{lantern1} fixed.

\begin{lem}\label{lem1}\cite{Lanier}
	Suppose $f$, $g$ and $h$ are elements of $\mathcal{M}(N_g)$ such that
	\[
	\begin{split}
	f(\alpha_3)&=\alpha_5\\
	g(\delta_1,\gamma)&=(\alpha_3,\alpha_5)\\
	h(\delta_2,\beta)&=(\alpha_3,\alpha_5)
	\end{split}
	\]
	and $f$, $g$ and $h$ preserve the orientation of neighbourhoods of these curves induced from $L$. 
Then the Dehn twist $t_{\alpha_1}$ may be written as a product of $f$, $g$ and $h$, an element conjugate to $f$, and their inverses.
\end{lem}
\begin{proof}
Since Dehn twists about nonintersecting curves commute, we can rearrange the lantern relation as
\begin{equation*}
t_{\alpha_1}=\left(t_{\alpha_3}t_{\alpha_5}^{-1}\right)\left(t_{\delta_1}t_{\gamma}^{-1}\right)\left(t_{\delta_2}t_{\beta}^{-1}\right)
\end{equation*}
Applying the assumptions on $g$ and $h$ gives us
\begin{equation*}
t_{\alpha_1}=\left(t_{\alpha_3}t_{\alpha_5}^{-1}\right)\left(g^{-1}\left(t_{\alpha_3}t_{\alpha_5}^{-1}\right)g\right)\left(h^{-1}\left(t_{\alpha_3}t_{\alpha_5}^{-1}\right)h\right);
\end{equation*}
applying the assumption on $f$ and regrouping gives us
\begin{equation*}
\begin{split}
t_{\alpha_1}&=\left(\left(f^{-1}t_{\alpha_5}f\right)t_{\alpha_5}^{-1}\right)
\left(g^{-1}\left(f^{-1}t_{\alpha_5}f\right)t_{\alpha_5}^{-1}g\right)
\left(h^{-1}\left(f^{-1}t_{\alpha_5}f\right)t_{\alpha_5}^{-1}h\right)\\
&=\left(f^{-1}\left(t_{\alpha_5}ft_{\alpha_5}^{-1}\right)\right)
\left(g^{-1}f^{-1}\left(t_{\alpha_5}ft_{\alpha_5}^{-1}\right)g\right)
\left(h^{-1}f^{-1}\left(t_{\alpha_5}ft_{\alpha_5}^{-1}\right)h\right).
\end{split}
\end{equation*}
We have written $t_{\alpha_1}$ as a product in $f$, $g$, $h$, $t_{\alpha_5}ft_{\alpha_5}^{-1}$ and their inverses.
\end{proof}

We close Preliminaries by proving two lemmas, which will be used to construct elements of $\M(N_g)$.

\begin{lem}\label{lem:lantchain}
Let $N=N_g$ be a closed nonorientable surface, $M_1,M_2\subset N$ two subsurfaces (not necessarily connected) and  $h:M_1\rightarrow M_2$ a homeomorphism. Assume that $N\setminus M_i$ is connected and nonorientable for $i=1,2$. Then there exists a homeomorphism $f:N\rightarrow N$ such that $f|_{M_1}=h$.
\end{lem}
\begin{proof}
Let $K_i=(N\setminus M_i)\cup \partial M_i$ for $i=1,2$. Then $K_i$ is a compact, connected and  nonorientable surface with boundary, $\partial K_i=\partial M_i$ and its Euler characteristic $\chi(K_i)=\chi(N)-\chi(M_i)$, so $\chi(K_1)=\chi(K_2)$. By the classification of compact surfaces two nonorientable surfaces with the same number of boundary components and equal Euler characteristic are homeomorphic, so $K_1$ and $K_2$ are homeomorphic. Let $\hat{h}:K_1\rightarrow K_2$ be a homeomorphism such that for any connected component $\gamma$ of $\partial K_1$ we have $\hat{h}(\gamma)=h(\gamma)$.
For each such component $\gamma\subset \partial K_1$ there exists a homeomorphism $\rho_\gamma\colon K_1\to K_1$ reversing orientation of $\gamma$ and equal to the identity on the other components of $\partial K_1$. Such $\rho_\gamma$ may obtained by pushing $\gamma$ once along a one-sided loop (see \cite[Lemma 5.1]{Szep_Osaka} or \cite{MK3}, where such a homeomorphism is called boundary slide). By composing $\hat{h}$ with $\rho_\gamma$ if necessary, and  may assume that $\hat{h}$ is equal to $h$ on $\partial K_1=\partial M_1$.
We can now define $f$ to be equal to $h$ on $M_1$ and $\hat{h}$ on $K_1$.
\end{proof}

\begin{lem}\label{lem:chains}
Suppose that $C_1=(\alpha_i)_{i=1}^n$ and $C_2=(\beta_i)_{i=1}^n$ are nonseparating $n$-chains on $N_g$, where either $g=n+1$ or $g=n+2$. There exists a homeomorphism $f\colon N_g\to N_g$ such that
$f(\alpha_i)=\beta_i$ for $i=1,\dots,n$.
\end{lem}

\begin{proof} For $i=1,2$ let $M_i$ be a regular neighbourhood of $C_i$ and $K_i=(N_g\setminus M_i)\cup\partial{M_i}$. Since $C_i$ is nonseparating, $K_i$ is connected.  Let $h\colon M_1\to M_2$ be a homeomorphism such that $h(\alpha_i)=\beta_i$ for $i=1,\dots,n$.

If $n$ is even, then $M_i$ is an orientable subsurface of genus $n/2$ with one boundary component for $i=1,2$ and thus $K_i$ must be nonorientable.  By Lemma \ref{lem:lantchain} there exists $f\colon N_g\to N_g$ such that $f|_{M_1}=h$.

Suppose now that $n$ is odd. Then $M_i$ is an orientable subsurface of genus $(n-1)/2$ with two boundary components for $i=1,2$ and the Euler characteristic of $K_i$ is
\[\chi(K_i)=\chi(N_g)-\chi(M_i)=1+n-g.\]
If $g=n+2$, then $\chi(K_i)=-1$ is and hence $K_i$ is nonorientable for $i=1,2$. Again, by Lemma \ref{lem:lantchain} there exists $f\colon N_g\to N_g$ such that $f|_{M_1}=h$.
If $g=n+1$ then for $i=1,2$, $\chi(K_i)=0$ and hence $K_i$ is an annulus which is glued to $M_i$ so that the resulting surface is nonorientable. It follows that $h$ can be extended to a homeomorphism $f\colon N_g\to N_g$.
 \end{proof}

\section{Three torsion generators of $\mathcal{M}(N_g)$}

In this section we prove Theorem \ref{thm1}.
	\begin{figure}\begin{center}
			(a)\includegraphics[scale=.3]{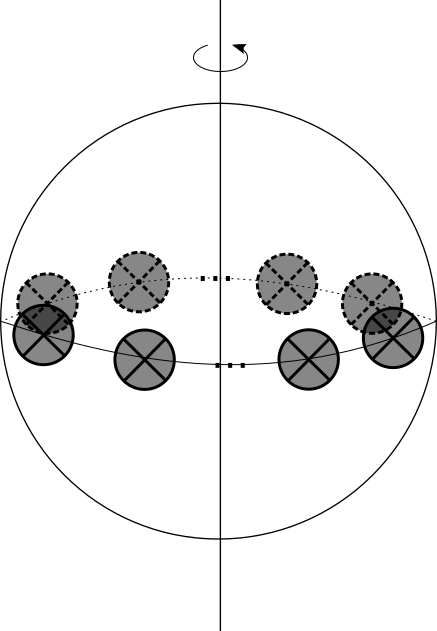}$\:\:\:\:\:\:$(b)
			\includegraphics[scale=.3]{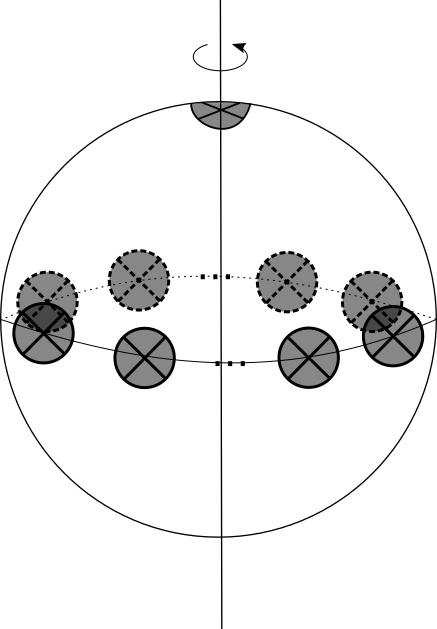}
			\caption{Rotation of order $k$  of (a) $N_k$ and (b) $N_{k+1}$}.
			\label{fig:nk}
		\end{center}
	\end{figure}
\begin{proof}[Proof of Theorem \ref{thm1}]
Let us define four torsion elements of $\M(N_g)$:
\begin{align*}
s&=t_{\alpha_1}t_{\alpha_2}\cdots t_{\alpha_{g-1}}, &r&=u_1u_2\cdots u_{g-1}\\
s'&=t_{\alpha_1}^2t_{\alpha_2}\cdots t_{\alpha_{g-1}}, &r'&=u_2\cdots u_{g-1}.
\end{align*}
The elements $r$ and $r'$ are conjugates of the rotations shown on Figure \ref{fig:nk} (respectively (a) and (b)) and they have orders $g$ and $g-1$ respectively. 
If $g$ is even, then $r\notin\T(N_q)$, whereas if $g$ is odd, then $r'\notin\T(N_g)$.

By the braid relations between Dehn twists we have $(s')^{g-1}=s^g$. If $g$ is even then $s$ and $s'$ have orders $g$ and $g-1$ respectively, whereas if $g$ is odd, they have orders $2g$ and $2(g-1)$ respectively (see \cite{PSz}). It is easy to check that $s(\alpha_i)=\alpha_{i+1}$ and $r(\alpha_i)=\alpha_{i+1}$ for $i=1,\dots,g-1$.
Let $y=t_{\alpha_{g-1}}u_{g-1}$. This is Lickorish's Y-homeomorphism, also called a crosscap slide (see \cite{PSz}).
 
The group $\M(N_1)$ is trivial \cite{Eps}, $\M(N_2)$ and $\T(N_2)$ are isomorphic respectively to $\mathbb{Z}_2\oplus\mathbb{Z}_2$ and $\mathbb{Z}_2$ \cite{Lick3}.
By the presentation of $\M(N_3)$ given in \cite{BC}, $\M(N_3)$ is generated by $t_{\alpha_1}$, $t_{\alpha_2}$ and $y$, and $\T(N_3)$ is generated by $t_{\alpha_1}$, $t_{\alpha_2}$. It follows that $\M(N_3)$ is generated by  $t_{\alpha_1}t_{\alpha_2}$, $t_{\alpha_1}^2t_{\alpha_2}$ and $u_2$ of orders $6$, $4$ and $2$ respectively, and $\T(N_3)$ is generated by $t_{\alpha_1}t_{\alpha_2}$ and $t_{\alpha_1}^2t_{\alpha_2}$ .

We will show that $\M(N_5)$ is generated by $s$, $st_{\beta}$ and $y^{-1}r'y$	.
Note that $st_{\beta}$ has order $6$. Let $G=\langle s, st_{\beta}, y^{-1}r'y\rangle$.
We have $t_{\beta}\in G$,  $t_{\alpha_4}=(st_{\beta})^{-1}t_{\beta}(st_{\beta})\in G$, and since all the curves $\alpha_i$ are in one $\langle s\rangle$-orbit, $t_{\alpha_i}\in G$ for $i=1,\dots,4$. We also have
	\[y^{-1}r'y(\alpha_2)=y^{-1}r'(\alpha_2)=y^{-1}(\alpha_3)=\varepsilon,\]
and hence $\T(N_5)\le G$ by Theorem \ref{thm:Genki}. Since $xr'x^{-1}\notin\T(N_5)$  we have $G=\M(N_5)$.

Now let $g\ge 6$. We define 
\[x=\begin{cases}
y^{-1}t_{\alpha_2}t_{\alpha_3}t_{\alpha_4}t_{\beta}&\text{if $g=6$}\\
t_{\alpha_{g-1}}u_{g-2}t_{\alpha_2}t_{\alpha_3}t_{\alpha_4}t_{\beta}&\text{if $g\ge 7$}.
\end{cases}
\]  
It can be checked that:
\[x(\alpha_4)=\beta,\ x(\alpha_2)=\alpha_3,\ x(\alpha_3)=\varepsilon\text{ if }g=6,\ x(\alpha_{g-1})=\varepsilon\text{ if }g\ge 7.\]	

Suppose that $g$ is even, $g\geq 6$ and let $G$ be the subgroup of $\M(N_g)$ generated by $s$, $s'$ and $xrx^{-1}$. We will show that $G=\M(N_g)$.  We have $t_{\alpha_1}=(s')s^{-1}\in G$, and since all the curves $\alpha_i$ are in one $\langle s\rangle$-orbit, $t_{\alpha_i}\in G$ for $i=1,\dots,g-1$.
	We have
	\begin{align*}
	&xr^2x^{-1}(\alpha_3)=xr^2(\alpha_2)=x(\alpha_4)=\beta,\\
	&xrx^{-1}(\alpha_3)=xr(\alpha_2)=x(\alpha_3)=\varepsilon\quad\textrm{if\ }g=6,\\
	&xr^{g-3}x^{-1}(\alpha_{3})=xr^{g-3}(\alpha_2)=x(\alpha_{g-1})=\varepsilon\quad\textrm{if\ }g\ge 8.
	\end{align*}
	It follows that $\beta$, $\varepsilon$ and $\alpha_3$ are in the same $G$-orbit, and hence $t_{\beta}, t_{\epsilon}\in G$. By Theorem \ref{thm:Genki}, $G$ contains generators of $\T(N_g)$, thus $\T(N_g)\le G$, and since $xrx^{-1}\notin\T(N_g)$  we have $G=\M(N_g)$.

If $g$ is odd, $g\geq 7$, then by the same argument as above $\M(N_g)$ is generated by $s$, $s'$ and $xr'x^{-1}$ (note that $r'\notin\T(N_g)$).

By substituting $s$ for $r$ and $r'$ in the above arguments we obtain that $\T(N_g)$ is generated by $s$, $s'$ and $xsx^{-1}$ for $g\ge 6$, and $\T(N_5)$ is generated by $s$, $st_{\beta}$ and $y^{-1}sy$.
\end{proof}

\section{Three conjugate torsion generators of $\mathcal{M}(N_g)$}
In this section we prove Theorem \ref{thm2}. We begin by constructing an element of given order $k$ for sufficiently large $g$.

	\begin{figure}\begin{center}
	(a)\includegraphics[scale=.35]{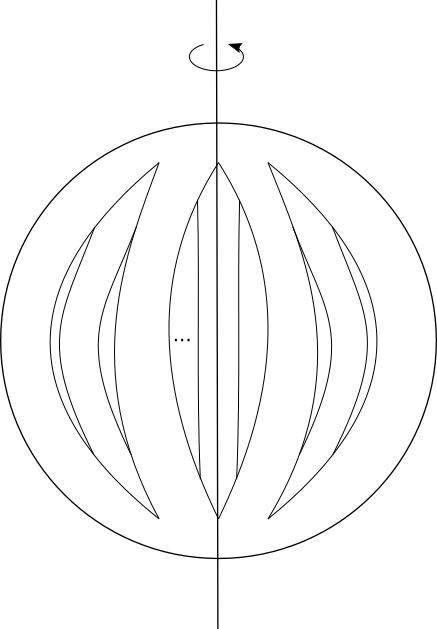}
		(b)\includegraphics[scale=.35]{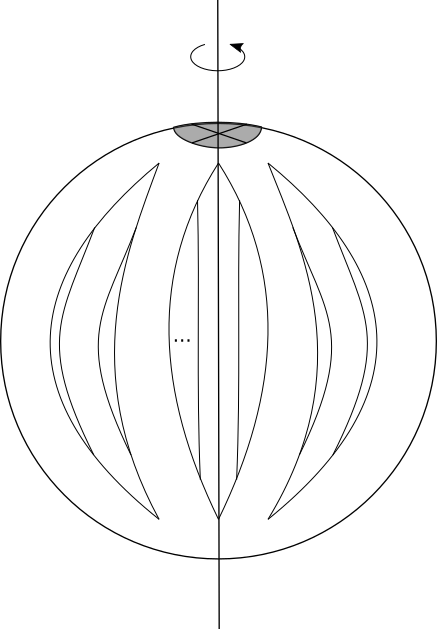}
	\caption{Rotation of order $k$ of (a) an orientable surface of genus $k-1$ and (b) a nonorientable surface of genus $2k-1$.}
	\label{fig:embed}	
\end{center}
\end{figure}
\begin{lem}\label{embed}
Let $k \geq 2$ and $g=pk+2q(k-1)$ or $g=pk+2q(k-1)+1$, where $p,q\in\mathbb{Z}_{\geq 0}$ and $p\ge 1$. Then $\mathcal{M}(N_g)$ contains an element of order $k$, which is in the twist subgroup if and only if either  $k$ is odd or $p$ is even.
\end{lem}
\begin{proof}
In \cite{Lanier}, Lanier shows a way to embed an orientable surface of genus $k-1$ into $\mathbb{R}^3$ so that it has rotational symmetry of order $k$: arrange two spheres along an axis of rotation and remove $k$ disks from each sphere, evenly spaced along the equator of each, and connect pairs of boundary components, one from each sphere, with a cylinder (Figure \ref{fig:embed}(a)). This can be done symmetrically so that a rotation by $2\pi/k$ permutes the cylinders cyclically.
	It is also easy to construct a model of a nonorientable surface of genus $k$ so that it has rotational symmetry of order $k$: arrange the $k$ crosscaps evenly along the equator of a sphere in a way that a rotation by $2\pi/k$ permutes the crosscaps cyclically (Figure \ref{fig:nk}(a)).

We can now use these two models to construct a model $E_g$ of a nonorientable surface of genus $g$ that also has a rotational symmetry of order $k$.
For $g=pk+2q(k-1)$  we define $E_g$ to be the connected sum of $p$ copies of a nonorientable surface of genus $k$ and $q$ copies of an orientable surface of genus $k-1$ along their axis of rotation (Figure \ref{fig:exembed} (a) or (b)). For $g=pk+2q(k-1)+1$ we take the connected sum as above and then add one crosscap, so that it is preserved by the rotation (Figure \ref{fig:exembed} (c) or (d)). 
We denote by $r$ the rotational symmetry of $E_g$.
\begin{figure}\begin{center}
\includegraphics[scale=.22]{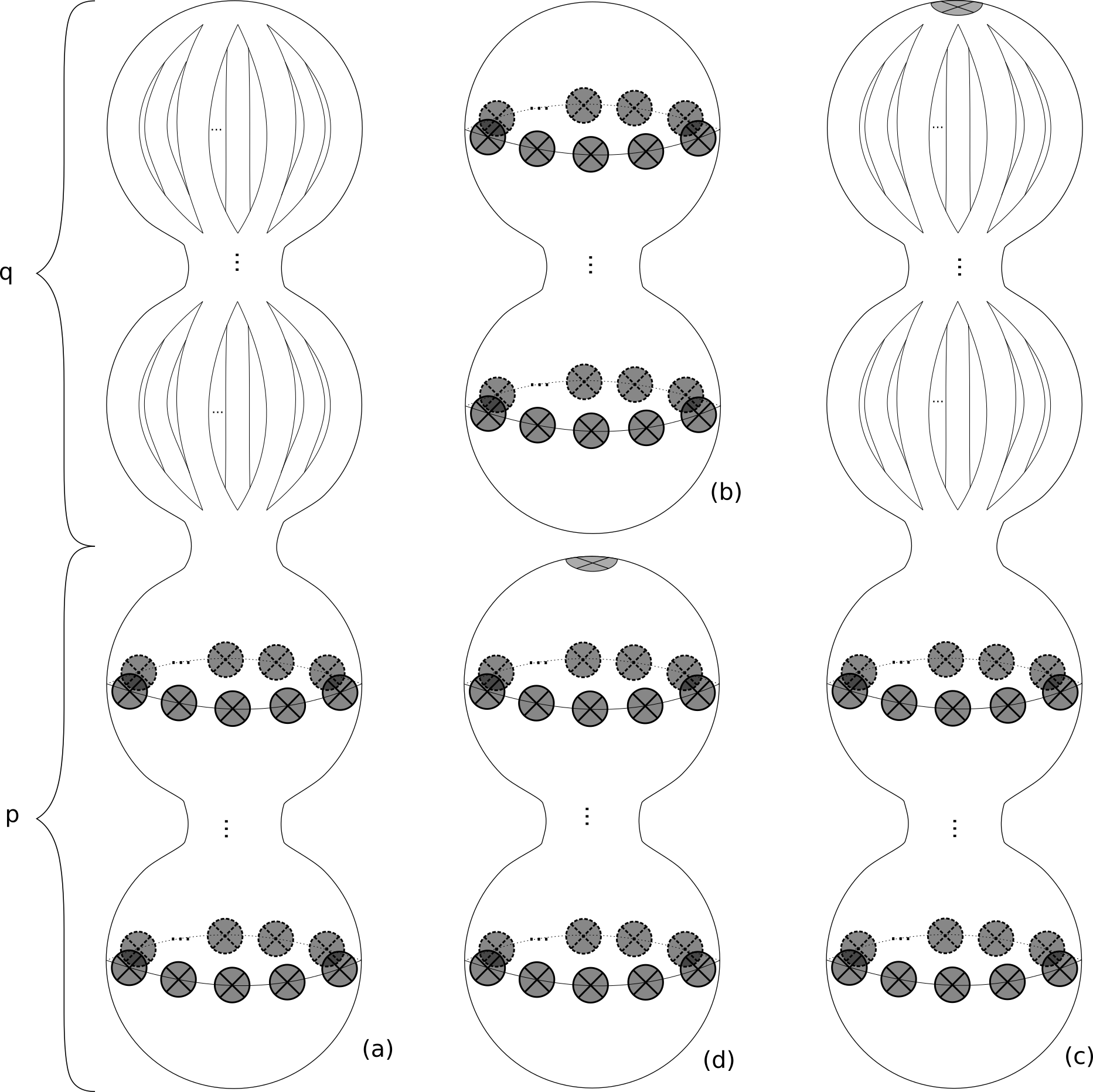}
	\caption{Models of $N_g$ with rotational symmetry of order $k$.}
	\label{fig:exembed}\end{center}	
\end{figure}

Since $\T(E_g)$ is a subgroup of index two of $\M(E_g)$, it contains all elements of odd order. Thus $r\in\T(E_g)$ if $k$ is odd. 

For the rest of this proof assume that $k$ is even. In order to apply Theorem \ref{thm:twsb}, we need to compute  $\det r_\ast$, where $r_\ast$ is the induced automorphism of $H_1(E_g,\mathbb{R})$. We assume that $g$ is even. For odd $g$ the computation is similar.

 We begin by setting a generating set of $H_1(E_g;\mathbb{R})$ composed of standard generators of the homology groups of the $p+q$ summands of the connected sum represented by curves shown in Figure \ref{fig:homembed}. Here we denote a curve and its homology class by the same symbol.
 
\begin{figure}\centering
\includegraphics[scale=.4]{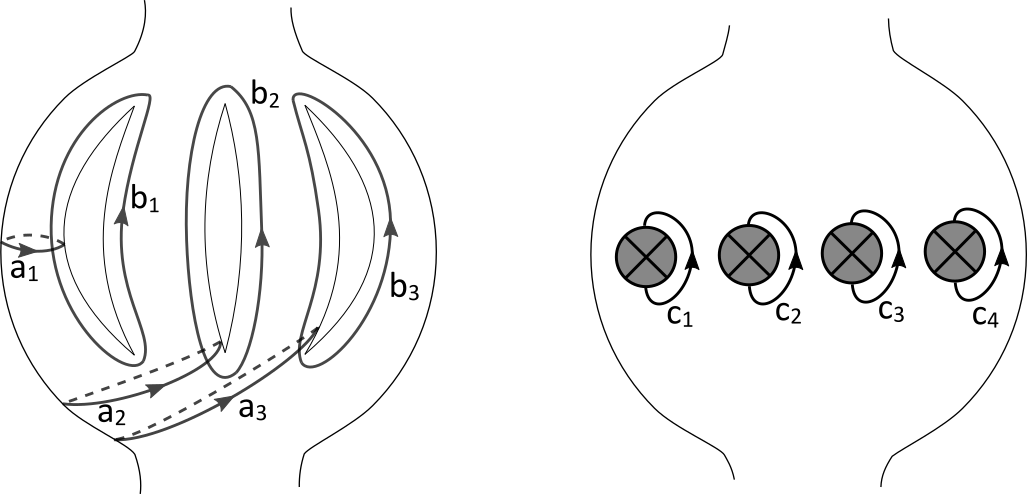}
	\caption{Generators of $H_1(E_g,\mathbb{R})$.}
	\label{fig:homembed}
\end{figure}

For the sake of simplicity, let us number these curves on each of the $r$-invariant subsurfaces separately: $a_i$ and $b_i$ for $i=1,\dots,k-1$ on each of the $q$ orientable subsurfaces of genus $k-1$ and one-sided curves $c_i$ for $i=1,\dots,k$ on each of the $p$ nonorientable subsurfaces of genus $k$.
We have
\begin{align*}
&r_\ast(a_i)=a_{i+1}-a_1,\ r_\ast(b_i)=b_{i+1}\quad\textrm{for\ }i=1,\dots,k-2,\\
&r_\ast(a_{k-1})=-a_1,\ r_\ast(b_{k-1})=-(b_1+\dots+b_{k-1}),\\
&r_\ast(c_i)=c_{i+1}\quad \textrm{for\ }i=1,\dots,k-1,\quad r_\ast(c_k)=c_1.
\end{align*} 

Note that this generating set consisting of $pk+2q(k-1)=g$ elements is not a basis of  $H_1(E_g,\mathbb{R})$, as the sum of all the $c_i$ generators is $0$. In order to obtain a basis we need to drop one of the $c_i$ generators. A straightforward computation gives  $\mathrm{det}(r_*)=(-1)^{p}$, which means, by Theorem \ref{thm:twsb}, that $r\in \mathcal{T}(E_g)$ if and only if $p$ is even.
\end{proof}

Now we prove Theorem \ref{thm2}. 
\begin{proof}[Proof of Theorem \ref{thm2}.]
Let $g=pk+2q(k-1)$ or $g=pk+2q(k-1)+1$.
We are going to construct three elements $f,g,h\in\M(N_g)$ such that:
\begin{itemize}
\item[(1)] $f$ and $g$ are conjugate to the rotation $r$ defined in Lemma \ref{embed} and  $h$ is a product of a power of $f$ and a power of $g$;
\item[(2)]  $f,g,h$ or their powers satisfy the assumptions of Lemma \ref{lem1};
\item[(3)]  all curves from Theorem \ref{thm:Genki} are in one $\langle f,g\rangle$-orbit.
\end{itemize}
Suppose that $f,g,h$ are such elements.
Then $f,g$ have order $k$. Let $H=\langle f,g,t_{\alpha_5}ft_{\alpha_5}^{-1}\rangle$. 
 By (2) and by the proof of Lemma \ref{lem1}, $t_{\alpha_1}\in H$, and by (3) $H$ contains generators of $\T(N_g)$. Since $\T(N_g)$ has index two in $\M(N_g)$, either $H=\T(N_g)$ or $H=\M(N_g)$. Furthermore, $H=\M(N_g)\iff f\notin\T(N_g)$, and by Lemma \ref{embed} $f\notin\T(N_g)$ if and only if $k$ is even and $p$ is odd.

\medskip
\noindent\textbf{Case1: $p=1$ and $q=0$.} 

Then $k=g$ or $k=g-1$ and $r$ is simply one of the rotations shown on Figure \ref{fig:nk} and, up to conjugation,  it can be written as \[r=u_1u_2\dots u_{k-1}.\] 
Let $f=r$.  
We have $f(\alpha_i)=\alpha_{i+1}$ for $i=1,\dots,k-2$, where $\alpha_i$ are the curves from Theorem \ref{thm:Genki} (Figure \ref{fig:genkicurves}). Thus $f^2$ maps $\alpha_3$ to $\alpha_5$. 

\begin{figure}\begin{center}
			\includegraphics[scale=.4]{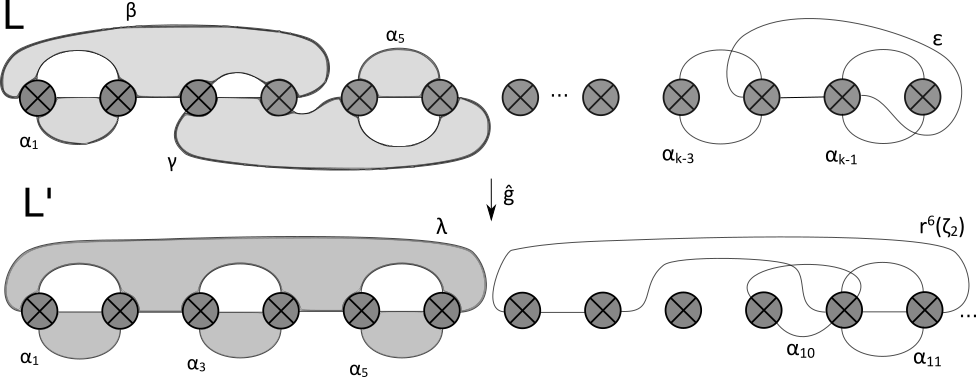}
			\caption{Map $\hat{g}$ on $N_k$.}
			\label{fig:nkg}
		\end{center}
	\end{figure}

Let $L\subset N_g$ be the four-holed sphere bounded by the curves $\alpha_1$, $\beta$, $\alpha_5$ and $\gamma$ shown on Figure \ref{lantern1}, and let $L'\subset N_g$ be the four-holed sphere bounded by the curves $\lambda$, $\alpha_1$, $\alpha_3$ and $\alpha_5$ shown on Figures \ref{fig:nkg} and \ref{fig:nkgdet} oriented so that the twists $t_{\alpha_i}$, $i=1,3,5$ are right-handed. We have the following lantern relation between right-handed twists on $L'$:
\[t_{\zeta_1}t_{\zeta_2}t_\beta=t_{\alpha_1}t_{\alpha_3}t_{\alpha_5}t_{\lambda}.\]
Let $\hat{g}_L\colon L\to L'$ be an orientation-preserving homeomorphism mapping
\[(\beta,\alpha_5,\gamma,\alpha_1,\alpha_3)\mapsto (\alpha_1,\alpha_3,\alpha_5,\lambda,\zeta_1).\]
The curves $\delta_1$ and $\delta_2$ on $L$ defined as $\delta_1=\hat{g}_L^{-1}(\zeta_2)$ and 
 $\delta_2=\hat{g}_L^{-1}(\beta)$, together with $\alpha_3$, $\gamma$, $\beta$, $\alpha_1$ and $\alpha_5$ form the configuration shown on the left-hand side of Figure \ref{lantern1}.
 	\begin{figure}\begin{center}
			\includegraphics[scale=.4]{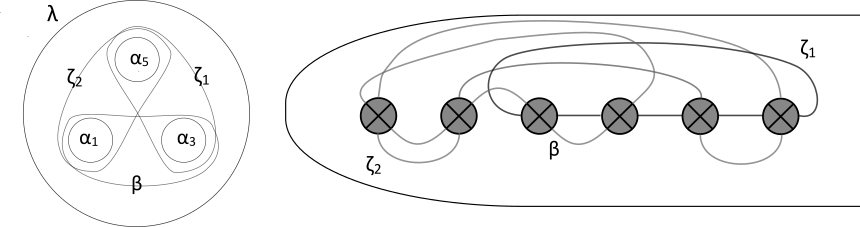}
			\caption{The curves of the lantern relation on $L'$.}
			\label{fig:nkgdet}
		\end{center}
	\end{figure}

Suppose that $g=k$ and let $M$ and $M'$ be regular neighbourhoods of the $3$-chains $(\alpha_{k-3},\varepsilon,\alpha_{k-1})$ and $(r^6(\zeta_2),\alpha_{10},\alpha_{11})$, respectively (see Figure \ref{fig:nkg}), oriented so that the twists $t_{\alpha_i}$ are right-handed. Let $\hat{g}_M\colon M\to M'$ be an orientation-preserving homeomorphism mapping
\[(\alpha_{k-3},\epsilon,\alpha_{k-1})\mapsto (r^6(\zeta_2),\alpha_{10},\alpha_{11}).\]
Since $N_k\setminus(L\cup M)$ and $N_k\setminus(L'\cup M')$ are connected and nonorientable,  by Lemma \ref{lem:lantchain}
there exists a homeomorphism $\hat{g}:N_k\rightarrow N_k$  equal to $\hat{g}_L$ on $L$ and $\hat{g}_M$ on $M$.

We define $g$ to be the mapping class of $\hat{g}^{-1}r\hat{g}$. Observe that $g^{2}$ maps $(\delta_2,\beta)$ to $(\alpha_3,\alpha_5)$, $g^6$ maps $(\delta_1,\gamma)$ to $(\alpha_{k-3},\alpha_{k-1})$ and $f^{6-k}$ maps $(\alpha_{k-3},\alpha_{k-1})$ to $(\alpha_3,\alpha_5)$. We define $h=f^{6-k}g^6$; then $h$ maps $(\delta_1,\gamma)$ to $(\alpha_3,\alpha_5)$.

Now $f^2$, $g^2$ and $h$ satisfy the assumptions of Lemma \ref{lem1}. Clearly all the curves $\alpha_i$ for $i=1,\dots,{k-1}$ are in one $\langle f\rangle$-orbit, and by the definition of $g$, $g^2(\beta)=\alpha_5$ and $g(\varepsilon)=\alpha_{k-1}$, meaning that all the curves from Theorem \ref{thm:Genki} are in one $\langle f,g \rangle$-orbit. 
	\begin{figure}\begin{center}
			\includegraphics[scale=.4]{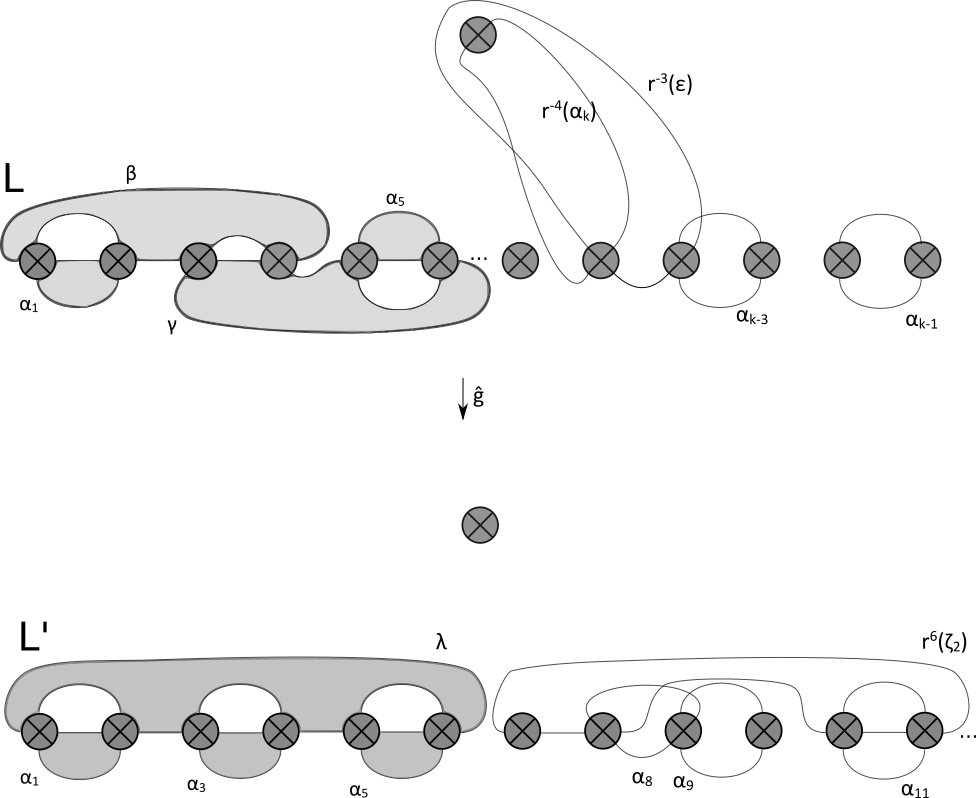}
			\caption{Map $\hat{g}$ on $N_{k+1}$.} 
			\label{fig:nk+1g}
		\end{center}
	\end{figure}

Now suppose that $g=k+1$. Let $M$ be the union of regular neighbourhoods of the $3$-chain $(r^{-4}(\alpha_k),r^{-3}(\epsilon),\alpha_{k-3})$ and the curve $\alpha_{k-1}$, and  let $M'$ be the union of regular neighbourhoods of the $3$-chain $ (\alpha_9,\alpha_{8},r^6(\zeta_2))$ and the curve $\alpha_{11}$
(see Figure \ref{fig:nk+1g}), oriented so that the twists $t_{\alpha_i}$ are right-handed. Let $\hat{g}_M\colon M\to M'$ be an orientation-preserving homeomorphism mapping 
\[(r^{-4}(\alpha_k),r^{-3}(\epsilon),\alpha_{k-3},\alpha_{k-1})\mapsto (\alpha_9,\alpha_{8},r^6(\zeta_2),\alpha_{11}).\]
Since $N_{k+1}\setminus(L\cup M)$ and $N_{k+1}\setminus(L'\cup M')$ are connected and nonorientable,  by Lemma \ref{lem:lantchain}
there exists a homeomorphism $\hat{g}:N_{k+1}\rightarrow N_{k+1}$  equal to $\hat{g}_L$ on $L$ and $\hat{g}_M$ on $M$.

We define $g$ to be the mapping class of $\hat{g}^{-1}r\hat{g}$.
Analogously to the above, it is easy to see that the mapping classes $f^2$, $g^2$, $h=f^{6-k}g^6$ satisfy the assumptions of Lemma \ref{lem1}. Clearly all the curves $\alpha_i$ for $i=1,\dots,{k-1}$ are in one $\langle f\rangle$-orbit, and by the definition of $g$, $g^2(\beta)=\alpha_5$, and $g^3f^{-3}(\varepsilon)=\alpha_{k-1}$ and $g^2f^{-4}(\alpha_k)=\alpha_{k-1}$, meaning that all the curves from Theorem \ref{thm:Genki} are in one $\langle f,g \rangle$-orbit. 

\medskip
\noindent\textbf{Case 2: $p+q\geq 2$.}	We will construct $f$ and $g$ by mapping $N_g$ to $E_g$, performing the rotation $r$, and then mapping back to $N_g$. Recall that $E_g$ was constructed in Lemma \ref{embed} as the connected sum of $p+q$ surfaces (Figure \ref{fig:homembed}). 
We label by $\eta_1,\dots,\eta_{p+q}$ the subsurfaces of $E_g$ corresponding to the connected summands, from bottom to top of the appropriate model in Figure \ref{fig:exembed}.

Let $n=pk+2q(k-1)-1$ so that  either $n=g-1$ or $n=g-2$. Let $F=(\alpha_i)_{i=1}^n$ be the $n$-chain on $N_g$ consisting of the curves $\alpha_i$ from Theorem \ref{thm:Genki} (Figure \ref{fig:genkicurves}; note that $\alpha_{g-1}$ is not included in $F$ if $n=g-2$). We will define a homeomorphism $\hat{f}\colon N_g\to E_g$ taking $F$ to a specific $n$-chain $\widetilde{F}$ on $E_g$ by using Lemma \ref{lem:chains}.

\begin{figure}\centering
		(a)\includegraphics[scale=.35]{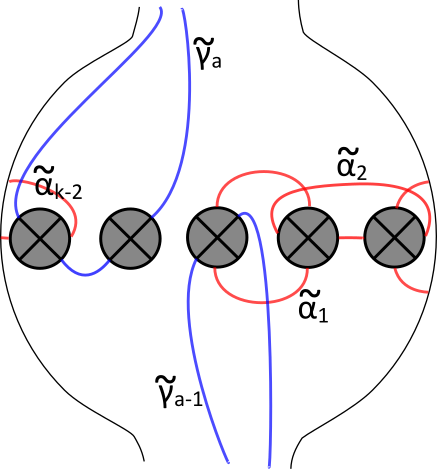}
	(b)\includegraphics[scale=.35]{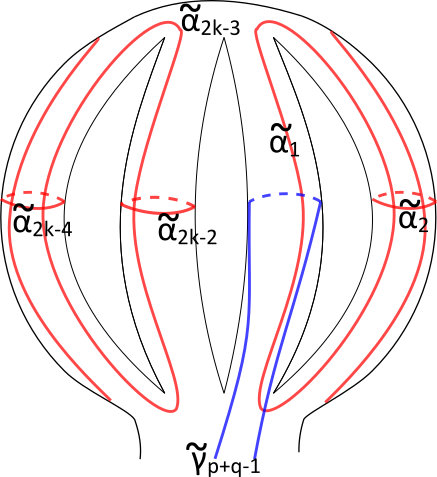}\\
	(c)\includegraphics[scale=.35]{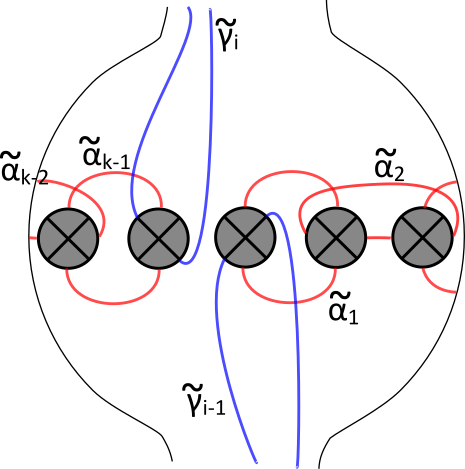}
	(d)\includegraphics[scale=.35]{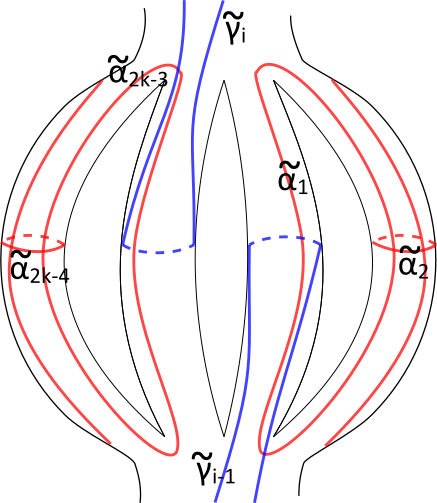}
	\caption{Chain of
		(a) $k-2$ curves in $\eta_{p}$,
		(b) $2k-2$ curves on $\eta_{p+q}$,
		(c) $k-1$ curves in $\eta_i$ for $1\leq i\leq p-1$,
		 (d) $2k-3$ curves in $\eta_i$ for $p+1\leq i\leq p+q-1$.}\label{fig:chainni}
\end{figure}

We begin by defining, for $i=1,\dots,p+q$,  a chain $\widetilde{F}_i=(\tilde{\alpha}_j)_{j=1}^{n_i}$ of curves on the subsurface $\eta_i$,  such that 
\begin{align*}
r(\tilde{\alpha}_j)&=\tilde{\alpha}_{j+1}\quad\textrm{for\ } 1\leq i\leq p\ \textrm{and\ }1\le j\le n_i-1,\\
r(\tilde{\alpha}_j)&=\tilde{\alpha}_{j+2}\quad\textrm{for\ } p+1\leq i\leq p+q\ \textrm{and\ }1\le j\le n_i-2.
\end{align*}
The chains $\widetilde{F}_i$ have lengths $n_i=k-1$ for $1\le i\le p-1$  (Figure \ref{fig:chainni}(c)). If $q=0$ then also $n_p=k-1$, otherwise $ n_p=k-2$ (see Figure \ref{fig:chainni}(a)), $n_i=2k-3$ for $p+1\leq i\leq p+q-1$ (Figure \ref{fig:chainni}(d)) and $n_{p+q}=2k-2$ (Figure \ref{fig:chainni}(b); if $g$ is odd, there is a crosscap on top of $\eta_{p+q}$).
We will use the labels $\tilde{\alpha}_j$ as ``local coordinates'' -- saying, for instance, ``the $\tilde{\alpha}_2$ on $\eta_3$''.

We also define  $p+q-1$ ``connecting curves'' $\widetilde{\gamma}_i$, $i=1,\dots,p+q-1$, each $\widetilde{\gamma}_i$ having geometric intersection number $1$ with the last curve of  $\widetilde{F}_i$ and the first curve of  $\widetilde{F}_{i+1}$ (see Figure \ref{fig:chainni}). The chains $\widetilde{F}_i$ together with the connecting curves $\widetilde{\gamma}_i$ form a nonseparating chain  $\widetilde{F}$ of length $n$. By Lemma \ref{lem:chains}, there exists a homeomorphism $\hat{f}\colon N_g\to E_g$ taking $F$ to $\widetilde{F}$.
Let $f$ be the mapping class of $\hat{f}^{-1}r\hat{f}$. 

For $i=1\dots p+q$ we define chains $F_i=\hat{f}^{-1}(\widetilde{F}_i)$ on $N_g$. We will call $\hat{f}^{-1}(\tilde{\gamma_i})$ the connecting curve between $F_i$ and $F_{i+1}$.

For $2\le i\le p$, we define $G_{i}$ to be the $3$-chain consisting of the last curve of $F_{i-1}$, the connecting curve between $F_{i-1}$ and $F_{i}$, and the first curve of $F_{i}$. 
For $p+1\le i\le p+q$ we define $G_i$ to consist of five curves: the last curve of $F_{i-1}$, the connecting curve between $F_{i-1}$ and $F_i$, first, second and fourth curves of $F_i$. Note that first four curves of $G_i$ form a chain from which the fifth curve is disjoint.

Recall from Case 1 the definition of $\hat{g}_L\colon L\to L'$, where $L$ and $L'$ are four-holed spheres.
Now we treat $L'$ as a subsurface of $\eta_1$ and define $\hat{g}:N_g\rightarrow E_g$ to be a homeomorphism equal to $\hat{g}_L$ on $L$ and mapping

\begin{align*}
(\alpha_7,\alpha_9)&\mapsto (r^6(\zeta_2),r^6(\alpha_5)) \mbox{ in } \eta_{1},\\
G_i &\mapsto (\tilde{\alpha}_1,\tilde{\alpha}_2,\tilde{\alpha}_3) \mbox{ in } \eta_{i},\;2\leq i\leq p\\
G_i &\mapsto (\tilde{\alpha}_1,\tilde{\alpha}_2,\tilde{\alpha}_3,\tilde{\alpha}_4,\tilde{\alpha}_7) \mbox{ in } \eta_{i},\;p+1\leq i\leq p+q\\
(\alpha_{g-3},\epsilon,\alpha_{g-1})&\mapsto \mbox{ last three }\tilde{\alpha_i}\mbox{ in } \eta_{p+q}.
\end{align*}
Similarly as in Case 1, such a map exists by Lemma \ref{lem:lantchain}.
Let $g$ be the mapping class of $\hat{g}^{-1}r\hat{g}$. As in Case 1, $g^2$ maps $(\delta_2,\beta)$ to $(\alpha_3,\alpha_5)$. Observe also that $g^6$ maps $(\delta_1,\gamma)$ to $(\alpha_7,\alpha_9)$ and $f^{-4}$ maps $(\alpha_7,\alpha_9)$ to $(\alpha_3,\alpha_5)$. We define $h=f^{-4}g^6$; then $h$ maps $(\delta_1,\gamma)$ to $(\alpha_3,\alpha_5)$. 

Now $f^2$, $g^2$ and $h$ satisfy the assumptions of Lemma \ref{lem1} and it remains to show that all the  curves from Theorem \ref{thm:Genki} are in the same $H$-orbit for $H=\langle f,g\rangle$.

First note that for $1\leq i\leq p$ all the curves of $F_i$  are in the same $\langle f\rangle$-orbit, while  for $p+1\leq i\leq p+q$ there are two $\langle f \rangle$-orbits of curves in $F_i$. However, since  $g^2$ maps the first curve of $F_i$ to the fourth one for $p+1\le i\le p+q$, all the curves of  each $F_i$ are in one $H$-orbit.  Since $g$ also maps between curves in different $F_i$, the connecting curves and $\alpha_{g-1}$ (which is not included in any $F_i$ for $g$ odd), all the $\alpha_i$ curves are on the same $H$-orbit. Finally, $h$ maps $\beta$ to $\alpha_5$ and a power of $g$ maps $\epsilon$ to a curve in $F_{p+q}$. 
Thus all the curves from Theorem \ref{thm:Genki} are in the same $H$-orbit.

\end{proof}

Proof of Corollary \ref{cor1}:
\begin{proof}
	By Theorem \ref{thm2} if $g=pk+2q(k-1)$ or $g=pk+2q(k-1)+1$ with $p$ odd, then there exists a set of three conjugate elements of order $k$ generating $\mathcal{M}(N_g)$. Let $g\geq 2(k-1)(k-2)+k$ and assume first that $g$ is even. Let $n=(g-k)/2$ and $m=\left\lfloor\frac{n}{k-1}\right\rfloor$. Then
	\[
n=m(k-1)+r\mbox{, } 0\leq r\leq k-2,
\]
\[
n=m(k-1)+r\left(k-(k-1)\right)=(m-r)(k-1)+rk.
\]
Notice that $m-r\geq 0$, as $n\geq (k-1)(k-2)$. Let $p=2r+1$ and $q=(m-r)$, then $g=2n+k=pk+2q(k-1)$.
For $g$ odd the argument is exactly analogous, except with $n=(g-k-1)/2$, and we obtain $g\geq 2(k-1)(k-2)+k+1$. This completes the proof.
\end{proof}

\end{document}